\documentclass[a4paper,reqno,11pt,final]{amsart}
\usepackage[oneside,DIV=12,pagesize=pdftex,paper=a4]{typearea}
\usepackage{amssymb, amsfonts, url, bm, %graphicx, wasysym, 
xspace}
\usepackage[notref,notcite]{showkeys}

\theoremstyle{plain}
\newtheorem{theorem}{Theorem}
\newtheorem{lemma}[theorem]{Lemma}
\newtheorem{corollary}[theorem]{Corollary}
\newtheorem{conjecture}[theorem]{Conjecture}
\newtheorem{proposition}[theorem]{Proposition}
\newtheorem{definition}[theorem]{Definition}

\theoremstyle{definition}

\newcommand{\setbuilder}[2]{\{#1\,\colon\,#2\}}
\newcommand{\set}[1]{\left\{#1\right\}}
\newcommand{\Bigset}[1]{\Bigl\{#1\Bigr\}}
\newcommand{\BMdist}[2]{d_{\mathrm{BM}}(#1,#2)}

\newcommand{\epsi}{\varepsilon}
\newcommand{\fhi}{\varphi}
\newcommand{\Tr}{\mathsf{T}}
\newcommand{\norm}[1]{\lVert#1\rVert}

\newcommand{\abs}[1]{\lvert#1\rvert}
\newcommand{\Bigabs}[1]{\Bigl\lvert#1\Bigr\rvert}

\newcommand{\card}[1]{\lvert#1\rvert}

\newcommand{\numbersystem}[1]{\mathbb{#1}}

\newcommand{\N}{\numbersystem{N}}

\newcommand{\R}{\numbersystem{R}}

\newcommand{\dimensional}{\nobreakdash-\hspace{0pt}dimensional\xspace}
\newcommand{\uproduct}[2]{#1\vee#2}

\newcommand{\vect}[1]{\bm{#1}}
\newcommand{\va}{\vect{a}}
\newcommand{\vb}{\vect{b}}
\newcommand{\vc}{\vect{c}}

\newcommand{\ve}{\vect{e}}

\newcommand{\vg}{\vect{g}}
\newcommand{\vh}{\vect{h}}

\newcommand{\vj}{\vect{j}}

\newcommand{\vo}{\vect{o}}
\newcommand{\vp}{\vect{p}}
\newcommand{\vq}{\vect{q}}

\newcommand{\vu}{\vect{u}}
\newcommand{\vv}{\vect{v}}

\newcommand{\vx}{\vect{x}}
\newcommand{\vy}{\vect{y}}
\newcommand{\vz}{\vect{z}}

\begin{document}

\bibliographystyle{amsplain}

\title{Maximal equilateral sets}
%\date{25\textsuperscript{th} August 2011}
\author{Konrad J.\ Swanepoel}
\thanks{Parts of this paper were written while the first author was at the Chemnitz University of Technology, and also during a visit to the Discrete Analysis Programme at the Newton Institute in Cambridge in May 2011.}
\address{Department of Mathematics,
London School of Economics and Political Science}
\email{k.swanepoel@lse.ac.uk}
\author{Rafael Villa}
\address{Departamento An\'alisis Matem\'atico, Facultad de Matem\'aticas,
Universidad de Se\-villa, c/Tarfia, S/N, 41012 Sevilla, Spain}
\email{villa@us.es}

\begin{abstract}
A subset of a normed space $X$ is called equilateral if the distance between any two points is the same.
Let $m(X)$ be the smallest possible size of an equilateral subset of $X$ maximal with respect to inclusion.
We first observe that Petty's construction of a $d$\dimensional $X$ of any finite dimension $d\geq 4$ with $m(X)=4$ can be generalised to give $m(X\oplus_1\R)=4$ for any $X$ of dimension at least $2$ which has a smooth point on its unit sphere.
By a construction involving Hadamard matrices we then show that for any set $\Gamma$, $m(\ell_p(\Gamma))$ is finite and bounded above by a function of $p$, for all $1\leq p<2$.
Also, for all $p\in[1,\infty)$ and $d\in\N$ there exists $c=c(p,d)>1$ such that $m(X)\leq d+1$ for all $d$\dimensional $X$ with Banach-Mazur distance less than $c$ from~$\ell_p^d$.
Using Brouwer's fixed-point theorem we show that $m(X)\leq d+1$ for all $d$\dimensional $X$ with Banach-Mazur distance less than $3/2$ from~$\ell_\infty^d$.
A graph-theoretical argument furthermore shows that $m(\ell_\infty^d)=d+1$.

The above results lead us to conjecture that $m(X)\leq 1+\dim X$ for all finite\dimensional normed spaces $X$.
\end{abstract}

\maketitle

\section{Introduction}\label{section:introduction}
Vector spaces in this paper are over the field $\R$ of real numbers.
Write $[d]:=\set{1,2,\dots,d}$ for any $d\in\N$ and $\binom{V}{k}:=\setbuilder{A\subseteq V}{\card{A}=k}$ for any set $V$ and $k\in\N$. 
Consider $d$\dimensional vectors to be functions $\vx:[d]\to\R$ denoted using the superscript notation $\vx=(\vx^{(1)},\dots,\vx^{(d)})$.
Similarly, write $\vx=(\vx^{(n)})_{n\in\Gamma}$ for any function $\vx\colon\Gamma\to\R$.
Write $\vo$ for zero vectors and zero functions.
For any $\gamma\in\Gamma$, let $\ve_\gamma$ denote the indicator function of $\set{\gamma}$, i.e., $\ve_\gamma(\gamma)=1$ and $\ve_\gamma(\delta)=0$ for all $\delta\in\Gamma\setminus\set{\gamma}$.
Given $\va=(\va^{(1)},\dots,\va^{(d)})\in\R^d$ and $\vb\in X$ with $X$ any vector space, define the \emph{Kronecker product} or \emph{tensor product} $\va\otimes \vb$ by $(\va^{(1)}\vb,\dots,\va^{(d)}\vb)\in X^d$.

Let $X$ denote a real normed vector space with norm $\norm{\cdot}=\norm{\cdot}_X$.
We also use \emph{space} or \emph{normed space} to refer to such spaces.
We will use the multiplicative Banach-Mazur distance between two normed spaces $X$ and $Y$ of the same finite dimension, denoted by $\BMdist{X}{Y}$ and defined to be the infimum of all $c\geq 1$ such that \[\norm{\vx}_X\leq\norm{T\vx}_Y\leq c\norm{\vx}_X\quad\text{for all $\vx\in X$}\]
for some invertible linear transformation $T\colon X\to Y$.

Let $\Gamma$ be any set.
For $p\in[1,\infty)$ let $\ell_p(\Gamma)$ denote the Banach space of all functions $\vx\colon\Gamma\to\R$ such that $\sum_{n\in\Gamma}\abs{\vx^{(n)}}^p < \infty$ with norm $\norm{\vx}_p=\left(\sum_{n\in\Gamma}\abs{\vx^{(n)}}^p\right)^{1/p}$.
Let $\ell_\infty(\Gamma)$ denote the Banach space of all bounded scalar-valued functions on $\Gamma$ with norm $\norm{\vx}_\infty:=\sup_{n\in\Gamma}\abs{\vx^{(n)}}$.
As usual, for any $p\in[1,\infty]$ we write $\ell_p$ for the sequence spaces $\ell_p(\N)$ and $\ell_p^d$ for $\ell_p([d])$.
If $X$ and $Y$ are two normed spaces, their $\ell_p$-sum $X\oplus_p Y$ is defined to be the direct sum $X\oplus Y$ with norm $\norm{(\vx,\vy)}_p:=\norm{(\norm{\vx}_X,\norm{\vy}_Y)}_p$.
%Also, write $c$ for the subspace of $\ell_\infty$ of convergent sequences, and $c_0$ for the subspace of null sequences.
Denote the \emph{sphere} and \emph{ball} in $X$ with center $\vc\in X$ and radius~$r>0$ by
\[S(\vc,r)=S_X(\vc,r):=\setbuilder{\vx\in X}{\norm{\vx-\vc}=r}\]
and
\[B(\vc,r)=B_X(\vc,r):=\setbuilder{\vx\in X}{\norm{\vx-\vc}\leq r}\text{,}\]
respectively.
See \cite{JL} for further background on the geometry of Banach spaces.
\begin{definition}
A subset $A\subseteq X$ is \emph{$\lambda$-equilateral} if $\norm{\vx-\vy}=\lambda$ for all $\set{\vx,\vy}\in \binom{A}{2}$.
A set $A\subseteq X$ is \emph{equilateral} if $A$ is $\lambda$-equilateral for some $\lambda>0$.
An equilateral set $A\subseteq X$ is \emph{maximal} if there does not exist an equilateral set $A'\subseteq X$ with $A\subsetneqq A'$.
\end{definition}
It is clear that a $\lambda$-equilateral set is a maximal equilateral set if and only if it does not lie on a sphere of radius $\lambda$.
Also, $A$ is $\lambda$-equilateral if and only if the balls $B(\vc,\lambda/2)$, $\vc\in A$, are pairwise touching.
It follows (as observed by Petty \cite{Petty} and P.~S.~Soltan \cite{Soltan}) by a result of Danzer and Gr\"unbaum \cite{DG} that an equilateral set in a $d$\dimensional normed space has cardinality at most $2^d$ with equality only if the unit ball is an affine cube, that is, only if the space is isometric to $\ell_\infty^d$, and in that case only if the set consists of all the vertices of some $\ell_\infty^d$-ball.
For a survey on equilateral sets, see \cite{Swanepoel2004}.
See also \cite{Swanepoel2008} for recent results on the existence of large equilateral sets in finite\dimensional spaces.
This paper will be exclusively concerned with maximal equilateral sets.
\begin{definition}
Let $m(X)$ denote the minimum cardinality of a maximal equilateral set in the normed space $X$.
\end{definition}
It follows from the above-mentioned result of Danzer and Gr\"unbaum that $m(X)\leq 2^d$ if $\dim X=d$.

We first dispose of the $2$\dimensional case.
By a simple continuity argument, any set of two points in a normed space of dimension at least $2$ can be extended to an equilateral set of size~$3$.
It is also possible to find a maximal equilateral set of size $3$ in any $2$\dimensional $X$.
In fact, if $X$ is not isometric to $\ell_\infty^2$ then by \cite{DG}, any equilateral set has cardinality $<4$, which implies that any equilateral set of size $3$ is already maximal.
Furthermore, by \cite{DG} the only equilateral sets of size $4$ consist of the vertices of an $\ell_2^\infty$-ball, so any equilateral set of size $3$ not consisting of three vertices of an $\ell_2^\infty$-ball (such sets exist by extending an appropriate equilateral set of $2$ points), is maximal.
%the result of Danzer and Gr\"unbaum gives that %Also, in $\ell_\infty^2$ it is easy to find a maximal equilateral set of size $3$.
It follows that $m(X)=3$ for all $2$\dimensional~$X$.

Now suppose that the dimension of $X$ is at least $3$.
Using a topological result, Petty \cite{Petty} showed that any equilateral set of size $3$ in $X$ can be extended to one of size $4$.
He also constructed, for each dimension $d\geq 3$, a $d$\dimensional normed space with a maximal equilateral set of size $4$.
Below (Proposition~\ref{pettygen}) we modify his example to show for instance that $\ell_1^d$ and $\ell_1$ also have this property.

A simple linear algebra argument shows that $m(\ell_2^d)=d+1$.
Brass \cite{Brass1999} and Dekster \cite{Dekster2000} independently showed that if $\BMdist{X}{\ell_2^d}\leq1+1/(d+1)$, then $m(X)\geq d+1$ (see also \cite[Theorem~8]{Swanepoel2004}).
By a theorem of Sch\"utte \cite{Schutte} (as pointed out by Smyth \cite{Smyth}) if $\BMdist{X}{\ell_2^d}\leq 1+1/(d+1)$ then $m(X)\leq d+1$.
In particular, using $\BMdist{\ell_p^d}{\ell_2^d}=d^{\abs{1/p-1/2}}$ (see for instance \cite{JL}) it follows that
\begin{equation}\label{+}
m(\ell_p^d)=d+1\quad\text{if}\quad \Bigabs{\frac{1}{p}-\frac{1}{2}}\leq\frac{1+o(1)}{d\ln d}\text{.}
\end{equation}

Even though $\ell_\infty^d$ has an equilateral set of size $2^d$, it turns out to have a maximal equilateral set of size $d+1$.
More generally, we show the following:
\begin{theorem}\label{thm:linfty}
If $\BMdist{X}{\ell_\infty^d}<3/2$, then $m(X)\leq d+1$.
In addition, $m(\ell_\infty^d)=d+1$.
\end{theorem}
Theorem~\ref{thm:linfty} will follow from Proposition~\ref{prop:linftybrouwer} in Section~\ref{section:brouwer} and Proposition~\ref{prop:linftygraph} in Section~\ref{section:graphs}.
A~similar result holds for the $\ell_p^d$ spaces, $1\leq p<\infty$.
\begin{theorem}\label{thm:lpall}
For all $p\in[1,\infty)$ and all $d\in\N$, $m(\ell_p^d)\leq d+1$.
\end{theorem}
\begin{theorem}\label{thm:approxlp}
For each $p\in(1,\infty)$ and $d\geq 3$ there exists $c=c(p,d)>1$ such that $m(X)\leq d+1$ for any $d$\dimensional $X$ with $\BMdist{X}{\ell_p^d}<c$.
\end{theorem}
Theorems~\ref{thm:lpall} and \ref{thm:approxlp} will be proved in Section~\ref{section:calculation}.
Our main result is a surprising property of $\ell_p$ where $1\leq p<2$.
It gives many examples of finite and infinite dimensional spaces with finite maximal equilateral sets.
These examples are essentially different from Petty's example alluded to above (which we also generalise in Proposition~\ref{pettygen} below).
\begin{theorem}\label{thm:lp}
For each $p\in[1,2)$ there exist $C=C(p)\in\N$ and $d_0=d_0(p)\in\N$ such that $m(\ell_p^d)\leq C$ for any $d\geq d_0$, and in fact, for any \textup{(}finite or infinite dimensional\textup{)} normed space $X$ and any $q\in[1,\infty)$, also $m(\ell_p^d\oplus_q X)\leq C$.

When $p\to 2$, we have $C(p)= O(1/(2-p))$ and $d_0(p) = O(1/(2-p))$.
Upper bounds are given in Table~\ref{table} for all~$p\in [1,\frac{\log (23/6)}{\log 2})$.
\end{theorem}
In particular, we obtain the following surprising corollary.
\begin{corollary}\label{cor:lp}
For any set $\Gamma$ and any $p\in[1,2)$, $m(\ell_p(\Gamma))$ is bounded above by a constant depending only on $p$.
\end{corollary}
The asymptotic bounds on $C(p)$ and $d_0(p)$ for $p\to 2$ in the above theorem are close to optimal, as \eqref{+} implies that \[C(p)=\Omega\left(\frac{1}{(2-p)\ln\left((2-p)^{-1}\right)}\right)\quad\text{and}\quad d_0(p)=\Omega\left(\frac{1}{(2-p)\ln\left((2-p)^{-1}\right)}\right)\text{.}\]
\begin{table}
\begin{center}
\begin{tabular}{r@{\;}c@{\;}l|c|c|l}
\multicolumn{3}{c|}{Range of $p$} & $C(p)$ & $d_0(p)$ & \qquad Proof\\[1mm] \hline &&&&&\\[-3mm]
$\displaystyle 1\leq$ & $p$ & $\displaystyle < \frac{\log (5/2)}{\log 2}\approx 1.32$ & $5$ & $4$ & Proposition~\ref{prop17}\\[4mm]
& $p$ & $\displaystyle = \frac{\log (5/2)}{\log 2}$ & $6$ & $4$ & Proposition~\ref{prop17}\\[4mm]
$\displaystyle \frac{\log (5/2)}{\log 2} <$ & $p$ & $\displaystyle < \frac{\log 3}{\log 2}\approx 1.58$ & $8$ & $6$ & Prop.~\ref{prop20} with $(k_1,k_2)=(2,2)$\\[4mm]
$\displaystyle \frac{\log 3}{\log 2} \leq$ & $p$ & $\displaystyle \leq \frac{\log (13/4)}{\log 2}\approx 1.70$ & $12$ & $10$ & Prop.~\ref{prop20} with $(k_1,k_2)=(2,4)$\\[4mm]
$\displaystyle \frac{\log (13/4)}{\log 2} <$ & $p$ & $\displaystyle < \frac{\log (7/2)}{\log 2}\approx 1.81$ & $16$ & $14$ & Prop.~\ref{prop20} with $(k_1,k_2)=(4,4)$ \\[4mm]
$\displaystyle \frac{\log (7/2)}{\log 2}\leq$ & $p$ & $\displaystyle \leq \frac{\log (29/8)}{\log 2}\approx 1.86$ & $24$ & $22$ & Prop.~\ref{prop20} with $(k_1,k_2)=(4,8)$ \\[4mm]
$\displaystyle \frac{\log (29/8)}{\log 2} <$ & $p$ & $\displaystyle < \frac{\log (15/4)}{\log 2}\approx 1.907$ & $32$ & $30$ & Prop.~\ref{prop20} with $(k_1,k_2)=(8,8)$ \\[4mm]
$\displaystyle \frac{\log (15/4)}{\log 2}\leq$ & $p$ & $\displaystyle \leq \frac{\log (91/24)}{\log 2}\approx 1.923$ & $40$ & $38$ & Prop.~\ref{prop20} with $(k_1,k_2)=(8,12)$ \\[4mm]
$\displaystyle \frac{\log (91/24)}{\log 2} <$ & $p$ & $\displaystyle < \frac{\log (23/6)}{\log 2}\approx 1.939$ & $48$ & $46$ & Prop.~\ref{prop20} with $(k_1,k_2)=(12,12)$ \\[4mm]
\end{tabular}\\[2mm]
\end{center}
\caption{Values of $C(p)$ and $d_0(p)$ in Theorem~\ref{thm:lp}}\label{table}
\end{table}
Theorem~\ref{thm:lp} and Corollary~\ref{cor:lp} will be proved in Section~\ref{section:hadamard} below.

We do not know of any $d$\dimensional space $X$ for which $m(X)>d+1$.
The above theorems give some evidence for the following conjecture:
\begin{conjecture}
 For any $d$\dimensional normed space $X$, $m(X)\leq d+1$.
\end{conjecture}

\section{A generalisation of Petty's example}\label{section:petty}
Petty \cite{Petty} showed that $m(\ell_2^d\oplus_1\R)=4$ for all $d\geq 2$.
In his argument $\ell_2^d$ can in fact be replaced by any, not necessarily finite\dimensional, normed space which has a \emph{smooth point} on its unit sphere, that is, a point where the norm is G\^ateaux differentiable, or equivalently, a point on the unit sphere which has only one supporting hyperplane \cite{JL, Roberts}.
By a classical theorem of Mazur \cite{Mazur1933} any separable normed space enjoys this property \cite[Theorem~10]{Roberts}.
\begin{proposition}\label{pettygen}
Let $X$ be a normed space of dimension at least $2$ with a norm which has a smooth point on its unit sphere.
Then $m(X\oplus_1\R)=4$.
\end{proposition}
\begin{proof}
Since $X\oplus_1\R$ is at least $3$\dimensional, $m(X\oplus_1\R)\geq 4$ by Petty's theorem mentioned in Section~\ref{section:introduction}.
For the upper bound, let $\vu\in X$ be a smooth point on the unit sphere of $X$.
Let $A:=\set{(\vo,1),(\vo,-1),(\vu,0),(-\vu,0)}$.
Then $A$ is a $2$-equilateral set in $X\oplus_1\R$.
If there exists a point $(\vx,r)\in X\oplus_1\R$ at distance $2$ to each point in $A$, then it easily follows that $r=0$, $\norm{\vx}=1$ and $\norm{\vx\pm\vu}=2$.
Then $\pm\vx$, $\pm\vu$ and $\pm\frac12\vx\pm\frac12\vu$ are all unit vectors in $X$ and by convexity the unit ball of the subspace $Y$ of $X$ generated by $\vu$ and $\vx$ is the parallelogram with vertices $\pm \vu$ and $\pm \vx$.
In particular, the unit ball of $Y$ has more than one supporting line at $\vu$, and so by the Hahn-Banach theorem, the unit ball of $X$ has more than one supporting hyperplane at~$\vu$.
\end{proof}

As special cases, $m(\ell_1)=m(\ell_1^d)=4$ for $d\geq 3$.
However, if $\Gamma$ is an uncountable set, then it is well known that no point on the unit sphere of $\ell_1(\Gamma)$ is smooth.
(This can be seen as follows:
Let $\vu\in\ell_1(\Gamma)$ have norm $1$.
Then $\vu\colon\Gamma\to\R$ has countable support $U\subset\Gamma$, say.
Choose any $i\in\Gamma\setminus U$.
Then $\norm{\vu\pm\ve_i}_1=2$ and the intersection of the unit ball of $\ell_1(\Gamma)$ with the subspace generated by $\vu$ and $\ve_i$ is the parallelogram with vertices $\pm\vu$ and $\pm\ve_i$, as in the proof of Proposition~\ref{pettygen}.)
Nevertheless, Theorem~\ref{thm:lp} gives the upper bound $m(\ell_1(\Gamma))\leq 5$ for any set $\Gamma$.

\section{Applying Brouwer's fixed point theorem}\label{section:brouwer}
\begin{proposition}\label{prop:linftybrouwer}
If $\BMdist{X}{\ell_\infty^d}<3/2$, then there exists a maximal equilateral set with $d+1$ elements. As a consequence, $m(X)\leq d+1$.
\end{proposition}
\begin{proof}
As preparation for the proof, we first exhibit a $2$-equilateral set $A$ of $d+1$ points in $\ell_\infty$ such that $S(\vo,1)$ is the unique sphere (of any radius) that passes through $A$.
For $i\in[d+1]$ and $n\in[d]$, let 
\[ \vp_{i}^{(n)}:=\begin{cases}
              -1 &\text{if $n=i$},\\
              0 &\text{if $n>i$},\\
              1 &\text{if $n<i$},\\
             \end{cases}
\]
and set $A=\set{\vp_1,\dots,\vp_{d+1}}$.
Suppose that $A\subset S(\vx,r)$ for some $\vx\in X$ and $r>0$.
Then for each $n\in[d]$, $\abs{x^{(n)}\pm 1}\leq r$, hence $\abs{x^{(n)}}\leq r-1$ and $r\geq 1$.
If we can show that $r=1$, we would also get $\vx=\vo$.
Suppose for the sake of contradiction that $r>1$.

We first show that $\vx=(r-1,r-1,\dots,r-1)$.
If not, let $m$ be the smallest index such that $\vx^{(m)}\neq r-1$.
Then for all $n<m$, $\abs{\vx^{(n)}-\vp_m^{(n)}}=\abs{r-1-1}<r$, and for $n>m$, $\abs{\vx^{(n)}-\vp_m^{(n)}}=\abs{x^{(n)}}\leq r-1$.
It follows that $r=\norm{\vx-\vp_m}_\infty=\abs{\vx^{(m)}+1}$.
Thus $\vx^{(m)}=-1\pm r$, which contradicts $\abs{x^{(n)}}\leq r-1$ and the choice of $m$.
Therefore, $\vx=(r-1,r-1,\dots,r-1)$.

Since $r=\norm{\vx-\vp_{d+1}}_\infty=\abs{r-1-1}<r$, we have obtained a contradiction.
Therefore, $A$ lies on a unique sphere.
As this sphere has radius $1$, $A$ is maximal equilateral.
This shows that $m(\ell_\infty^d)\leq d+1$.

We now prove the general result.
Let $D:=\BMdist{X}{\ell_\infty^d}<3/2$, and assume without loss of generality that $X=(\R^d,\norm{\cdot})$ such that
\begin{equation}\label{eq:norms}
 \norm{\vx} \leq \norm{\vx}_\infty \leq D\norm{\vx}\text{ for all }\vx\in\R^d.
\end{equation}
We will prove that $m(X)\leq d+1$ by finding a perturbation of the above set $A$ that will be maximal equilateral in $X$.
We use Brouwer's fixed-point theorem as in \cite{Brass1999} and \cite{Swanepoel2008}.
Consider the space $\R^{\binom{[d+1]}{2}}$ of vectors indexed by unordered pairs of elements from $[d+1]$.
Write $\vz^{\set{i,j}}$ for the coordinate of $\vz\in\R^{\binom{[d+1]}{2}}$ indexed by $\set{i,j}$.
Given $\vz\in I:=[0,1]^{\binom{[d+1]}{2}}\subset\R^{\binom{[d+1]}{2}}$, define $d+1$ points $\vp_1(\vz),\dots,\vp_{d+1}(\vz)\in\R^d$ as follows.
For $i\in[d+1]$ and $n\in[d]$, let
\begin{equation}\label{eq:pidef}
 \vp_{i}^{(n)}(\vz):=\begin{cases}
              -1 &\text{if $n=i$},\\
              0 &\text{if $n>i$},\\
              1+\vz^{\set{n,i}} &\text{if $n<i$}.\\
             \end{cases}
\end{equation}
Define the mapping $\fhi:I\to I$ by setting
\[ \fhi^{\set{i,j}}(\vz):=\norm{\vp_i(\vz)-\vp_j(\vz)}_\infty-\norm{\vp_i(\vz)-\vp_j(\vz)}=2+z^{\set{i,j}}-\norm{\vp_i(\vz)-\vp_j(\vz)}\]
for each $\set{i,j}\in\binom{[d+1]}{2}$.
Then by \eqref{eq:norms}, 
$\fhi^{\set{i,j}}(\vz)\geq 0$
and
\begin{align*}
 \fhi^{\set{i,j}}(\vz) &\leq \norm{\vp_i(\vz)-\vp_j(\vz)}_\infty-\frac{1}{D}\norm{\vp_i(\vz)-\vp_j(\vz)}_\infty\\
 &=\left(1-\frac{1}{D}\right)(2+\vz^{\set{i,j}})< \left(1-\frac{2}{3}\right)(2+1)=1\text{.}
\end{align*}
Thus $\fhi$ is well-defined.
It is clearly continuous, and so has a fixed point $\vz_0\in I$ by Brouwer's theorem:
\[  2+\vz_0^{\set{i,j}}-\norm{\vp_i(\vz_0)-\vp_j(\vz_0)}=\vz_0^{\set{i,j}}\quad\text{for all } \set{i,j}\in\binom{[d+1]}{2}. \]
Therefore, $\set{\vp_1(\vz_0),\dots,\vp_{d+1}(\vz_0)}$ is $2$-equilateral in $X$.

From now on, write $\vp_i$ for $\vp_i(\vz_0)$.
Suppose that $\set{\vp_1,\dots\vp_{d+1}}$ is not maximal equilateral. 
Then there exists $\vx\in X$ such that $\norm{\vx-\vp_i}=2$ for each $i\in[d+1]$.
We first show that all $\abs{\vx^{(n)}}<2$, then that all $\vx^{(n)}\geq 1$, and then obtain a contradiction.

By \eqref{eq:norms},
\[ 2\leq\norm{\vx-\vp_i}_\infty\leq 2D\quad\text{for each $i\in[d+1]$}.\]
In particular, $\abs{\vx^{(n)}-\vp_n^{(n)}}=\abs{\vx^{(n)}+1}\leq 2D$, which gives $\vx^{(n)}\leq 2D-1<2$ for all $n\in[d]$.
Also, $\abs{\vx^{(n)}-\vp_{n+1}^{(n)}}\leq 2D$, that is, $\abs{\vx^{(n)}-1-\vz_0^{\set{n,n+1}}}\leq 2D$, which gives \[\vx^{(n)}\geq 1+\vz_0^{\set{n,n+1}}-2D>-2\text{.}\]
It follows that $\abs{\vx^{(n)}}<2$ for all $n\in[d]$.

Next let $m$ be the smallest index such that $\vx^{(m)}<1$.
For all $n<m$,
\[ \vx^{(n)}-\vp_m^{(n)} = \vx^{(n)} - (1 + \vz_0^{\set{n,m}}) \geq -\vz_0^{\set{n,m}} \geq -1\]
and
\[ \vx^{(n)}-\vp_m^{(n)} < 2 - (1 + \vz_0^{\set{n,m}}) = 1 -\vz_0^{\set{n,m}} \leq 1\text{,}\]
hence $\abs{\vx^{(n)}-\vp_m^{(n)}}\leq 1$.
For all $n>m$, $\abs{\vx^{(n)}-\vp_m^{(n)}}=\abs{\vx^{(n)}} < 2$.
It follows that $2\leq\norm{\vx-\vp_m}_\infty=\abs{\vx^{(m)}+1}$.
However, $\vx^{(m)}+1<2$ by assumption and $\vx^{(m)}+1 > -2+1$, so we obtain a contradiction.

It follows that $\vx^{(n)}\geq 1$ for all $n\in[d]$.
Then $\abs{\vx^{(n)}-\vp_{d+1}^{(n)}} = \abs{\vx^{(n)} - 1 - \vz_0^{\set{n,d+1}}}\leq1$ since $1\leq \vx^{(n)}<2$ and $0\leq\vz_0^{\set{n,d+1}}\leq 1$.
It follows that $\norm{\vx-\vp_{d+1}}_\infty\leq 1$, a contradiction.
\end{proof}

\section{Using graphs}\label{section:graphs}
In their studies of neighbourly axis-parallel boxes, Zaks \cite{Zaks1985} and Alon \cite{Alon1997} modelled a certain geometric problem as a problem about covering a complete graph by complete bipartite subgraphs.
We use the same technique when showing that an arbitrary equilateral set of at most $d$ points in $\ell_\infty^d$ can be extended to a larger equilateral set.
Our proof in fact shows that any collection of at most $d$ pairwise touching, axis-parallel boxes in $\R^d$ can be extended to a pairwise touching collection of $d+1$ axis-parallel boxes.

%As usual, the edges of a graph are considered to be unordered pairs.
The graph-theoretical result needed is the following simple lemma which states in particular that if the edges of a complete graph on $k$ vertices are covered by at least $k$ complete bipartite subgraphs $G_i$, then for each $G_i$ we may choose one of its two parts such that the chosen parts cover all $k$ vertices.
For technical reasons we have to allow one, but not more than one, of the classes of the complete bipartite subgraphs to be empty.
%Let $K_k$ denote the complete graph with vertex set $[k]$ and edge set $\binom{[k]}{2}$.
Thus we define the \emph{join} of $A,B\subseteq [k]$ to be $\uproduct{A}{B}:=\setbuilder{\set{a,b}}{a\in A, b\in B}$ whenever $A\cap B=\emptyset$ and $A\cup B\neq\emptyset$.
%If $A,B\neq\emptyset$ then $\uproduct{A}{B}$ is the set of edges of a complete bipartite subgraph of $K_k$.
%As the definition implies, $\uproduct{A}{B}=\uproduct{B}{A}$.
\begin{lemma}\label{lemma:graph}
Let $d\geq k\geq 1$ be integers.
For each $n\in[d]$ let $A_n^0,A_n^1\subseteq[k]$ be given such that $A_n^0\cap A_n^1=\emptyset$ and $A_n^0\cup A_n^1\neq\emptyset$.
Suppose that $\bigcup_{n\in[d]}(\uproduct{A_n^0}{A_n^1})=\binom{[k]}{2}$.
Then there exist $\sigma_1,\dots,\sigma_d\in\set{0,1}$ such that $A_1^{\sigma_1}\cup\dots\cup A_d^{\sigma_d}=[k]$. 
\end{lemma}
\begin{proof}
We use induction on $k\in\N$.
The case $k=1$ is trivial, so we assume that $k\geq 2$ and that the theorem holds for $k-1$.
If for each $j\in[k]$, some $\uproduct{A_n^0}{A_n^1}$ equals $\uproduct{\emptyset}{\set{j}}$, take $\sigma_n$ such that $A_n^{\sigma_n}=\set{j}$ for each of these $n$.
Then choose all remaining $\sigma_n$ arbitrarily to obtain the required covering of $[k]$.

Thus assume without loss of generality that $\uproduct{\emptyset}{\set{k}}$ does not occur as a $\uproduct{A_n^0}{A_n^1}$.
Without loss of generality, $\set{1,k}\in\uproduct{A_d^0}{A_d^1}$ (note $k\geq 2$).
Thus $k\in A_d^{\sigma_d}$ for some $\sigma_d\in\set{0,1}$.
Set $B_n^0:=A_n^0\setminus\set{k}$ and $B_n^1:=A_n^1\setminus\set{k}$ for each $n\in[d-1]$.
Then $\bigcup_{n\in[d-1]}(\uproduct{B_n^0}{B_n^1})=\binom{[k-1]}{2}$.
Since all $\uproduct{A_n^0}{A_n^1}$ are different from $\uproduct{\emptyset}{\set{k}}$, we still have $B_n^0\cup B_n^1\neq\emptyset$, so we may apply the induction hypothesis to obtain $B_n^{\sigma_n}$, $n\in[d-1]$, with union $[k-1]$.
Together with $A_d^{\sigma_d}$ we have obtained the required covering of $[k]$. 
\end{proof}

\begin{proposition}\label{prop:linftygraph}
$m(\ell_\infty^d)\geq d+1$.
\end{proposition}
\begin{proof}
We show that any $1$-equilateral set $\set{\vp_1,\dots,\vp_k}\subset\ell_\infty^d$ of size at most $k\leq d$ can be extended.
Without loss of generality, $k\geq 1$.

Since $\abs{\vp_i^{(n)}-\vp_j^{(n)}}\leq 1$ for all $\set{i,j}\in\binom{[k]}{2}$ and $n\in[d]$, we may assume after a suitable translation that all $\vp_i\in[0,1]^d$.
For each $n\in[d]$, define $A_n^0:=\setbuilder{i}{\vp_i^{(n)}=0}$ and $A_n^1:=\setbuilder{i}{\vp_i^{(n)}=1}$.
Again by making a suitable translation we may assume that each $A_n^0\cup A_n^1\neq\emptyset$.

Since $\set{\vp_1,\dots,\vp_k}$ is $1$-equilateral, each $\set{i,j}\in\binom{[k]}{2}$ is in some $\uproduct{A_n^0}{A_n^1}$, $n\in[d]$.
Indeed, since $\norm{\vp_i-\vp_{j}}_\infty=1$, there exists an $n\in[d]$ with $\abs{\vp_i^{(n)}-\vp_{j}^{(n)}}=1$.
Since $0\leq \vp_i^{(n)}, \vp_{j}^{(n)} \leq 1$, it follows that $\set{\vp_i^{(n)}, \vp_{j}^{(n)}}=\set{0,1}$, which gives $\set{i,j}\in \uproduct{A_n^0}{A_n^1}$.

By Lemma~\ref{lemma:graph} $A_1^{\sigma_1}\cup\dots\cup A_d^{\sigma_d}=[k]$ for some $\sigma_1,\dots,\sigma_d\in\set{0,1}$.
Define $\vq=(1,1,\dots,1)-(\sigma_1,\dots,\sigma_d)$.
We show that for each $i\in[k]$, $\norm{\vp_i-\vq}_\infty = 1$.
Since $\vq\in[0,1]^d$, $\norm{\vp_i-\vq}_\infty\leq 1$.
There exists $n\in[d]$ such that $i\in A_n^{\sigma_n}$, i.e., $\vp_i^{(n)}=\sigma_n$.
It follows that $\abs{\vp_i^{(n)}-\vq^{(n)}}=1$, which gives $\norm{\vp_i-\vq}_\infty=1$.
\end{proof}

\section{A calculation}\label{section:calculation}
%We omit the simple proof of the following lemma.
Convexity of the function $x\mapsto\abs{x}^p$ for $p\geq 1$ readily implies the following lemma.
\begin{lemma}\label{lemma:calc}
For any $p\in[1,\infty)$ and $\lambda>0$ the function $f(x)=\abs{x+\lambda}^p-\abs{x}^p$, $x\in\R$, is increasing, and strictly increasing if $p>1$.
\end{lemma}

\begin{proposition}\label{prop:lpcalc}
For any $p \in (1,\infty)$ and $d\geq 3$, the $2^{1/p}$-equilateral set of standard unit vectors $S=\setbuilder{\ve_i}{i\in[d]}$ in $\ell_p^d$ can be extended in exactly two ways to equilateral sets $S\cup\set{\vp}$ and~$S\cup\set{\vq}$.
Furthermore, $\norm{\vp-\vq}_p>2^{1/p}$.
\end{proposition}
\begin{proof}
%The set of standard unit basis vectors $S=\set{\ve_1,\dots,\ve_d}$ in $\R^d$ is $2^{1/p}$-equilateral in $\ell_p^d$.
%We show that $S$ can be extended, and if $S$ is extended in two ways $S\cup\set{\vp}$ and $S\cup\set{\vq}$, then the distance $\norm{\vp-\vq}_p>2^{1/p}$.
%Thus both $S\cup\set{\vp}$ and $S\cup\set{\vq}$ will be maximal equilateral.
%(In fact $S$ has exactly two extensions, but we don't need this for the proof.)

Let $\vp$ be equidistant to all points of $S$, say $\norm{\vp -\ve_i}_p=c$ for all $i\in[d]$ where $c>0$ is fixed.
Then $\abs{\vp^{(i)}}^p-\abs{\vp^{(i)}-1}^p=\norm{\vp}_p^p-c^p$ for all $i$.
By Lemma~\ref{lemma:calc}, $\vp^{(1)}=\dots=\vp^{(d)}$, i.e., $\vp$ is a multiple of $\vj:=(1,1,\dots,1)\in\R^d$.

Suppose now $\vp=x\vj$ satisfies $\norm{\vp-\ve_i}_p=2^{1/p}$ for all $i\in[d]$.
It follows that
\begin{equation}\label{eq0}
\abs{x-1}^p+(d-1)\abs{x}^p=2.
\end{equation}
Consider the function $f(x)=\abs{x-1}^p+(d-1)\abs{x}^p$.
It is easily checked that $f$ has a unique minimum at a point $x_0\in(0,1)$ and is strictly decreasing on $(-\infty,x_0)$ and strictly increasing on $(x_0,\infty)$.
Since $f(-1)=2^p+d-1>2$ and $f(0)=1$ and $f(1)=d-1\geq 2$, equation~\eqref{eq0} has a unique negative solution $x=-\mu\in(-1,0)$ and a unique positive solution $x=\lambda\in(0,1]$.
%It is clearly strictly decreasing on $(-\infty,0]$, and since $f(0)=1$ and $f(-1)>2$, equation \eqref{eq0} has a unique negative solution $-\mu$, say, in the interval $(-1,0)$.
%Let $\lambda$ be any other solution to \eqref{eq0}.
%Then $\lambda>0$ (there is in fact a unique positive solution to \eqref{eq0}, but we don't need to show this), and w
We have to show that $\norm{-\mu\vj-\lambda\vj}_p>2^{1/p}$, that is, $\lambda+\mu>(2/d)^{1/p}$.
Since $\lambda$ is a solution to \eqref{eq0}, it follows that $2=(1-\lambda)^p+(d-1)\lambda^p<1+d\lambda^p$, hence $\lambda>(1/d)^{1/p}$.
It remains to show that $\mu\geq(2^{1/p}-1)/d^{1/p}$.
Suppose to the contrary that
\begin{equation}\label{eq6}
 \mu<\frac{2^{1/p}-1}{d^{1/p}}.
 \end{equation}
Since $x=-\mu$ is a solution of \eqref{eq0},
\begin{align*}
2 &=(1+\mu)^p+(d-1)\mu^p\\
&< (1+\mu)^p-\mu^p+(2^{1/p}-1)^p\quad\text{by \eqref{eq6}},
\end{align*}
hence
\[ (2^{1/p}-1+1)^p-(2^{1/p}-1)^p < (\mu+1)^p-\mu^p. \]
By Lemma~\ref{lemma:calc}, $2^{1/p}-1<\mu$, which contradicts \eqref{eq6}.
\end{proof}
\begin{proof}[Proof of Theorem~\ref{thm:lpall}]
We have already observed in Section~\ref{section:introduction} that $m(X)=3$ for any two\dimensional~$X$, so we may assume that $d\geq 3$.
We have also observed in Section~\ref{section:petty} that $m(\ell_1^d)= 4$ for all $d\geq 3$, so we may assume that $p\in(1,\infty)$.
Then the theorem follows from Proposition~\ref{prop:lpcalc}.
\end{proof}
\begin{proposition}\label{prop:lpapprox}
Let $1<p<\infty$, $d\geq 3$, $0<\epsi\leq(2d-4)^{-1/(p-1)}$, and $R=(1+\frac{p-1}{2}\epsi)^{1/p}$.
Suppose that $X=(\R^d,\norm{\cdot})$ is given such that
\[ \norm{x}\leq\norm{x}_p\leq R\norm{x}\quad\text{for all $\vx\in\R^d$.}\]
Then $X$ has a $\lambda$-equilateral set $\set{\vp_1,\dots,\vp_d}$, where $\lambda=\left(2+(d-2)\epsi^p\right)^{1/p}$, such that $\vp_i^{(i)}=1$ for all $i\in[d]$, $-\epsi<\vp_i^{(j)}<0$ for all $i,j\in[d]$ with $j<i$, and $\vp_i^{(j)}=0$ for all $i,j\in[d]$ with~$j>i$.
\end{proposition}
\begin{proof}
Let $\beta,\gamma>0$ be arbitrary (to be fixed later).
For $i\in[d]$ define $\vp_i\colon\R^{\binom{[d]}{2}}\to\R^d$ by setting for each $n\in[d]$,
\[ \vp_i^{(n)}(\vz) = 
    \begin{cases}
        z^{\set{n,i}} & \text{if $n<i$,}\\
        -\gamma     & \text{if $n=i$,}\\
        0           & \text{if $n>i$.}
    \end{cases}
\]
That is,
\[ \vp_i(\vz) = (z^{\set{1,i}}, \dots, z^{\set{i-1,i}}, -\gamma, 0, \dots, 0). \]

Let $I=[0,\beta]^{\binom{[d]}{2}}$ and define $\fhi\colon I \to I$ by
\[ \fhi^{\set{i,j}}(\vz) = 1+z^{\set{i,j}} - \norm{\vp_i(\vz)-\vp_j(\vz)}\quad \text{for each $\set{i,j}\in\binom{[d]}{2}$.}\]
It is clear that $\fhi$ is continuous.
We next show that $\fhi$ is well defined if $R$, $\beta$, and $\gamma$ are chosen appropriately.
Let $\vz\in I$.
Then $0\leq z^{\set{i,j}}\leq\beta$ for all $\set{i,j}\in\binom{[d]}{2}$.
We first bound~$\norm{\vp_i(\vz)-\vp_j(\vz)}_p$.
Without loss of generality, $i<j$.
Then
\begin{align}
\norm{\vp_i(\vz)-\vp_j(\vz)}_p^p &= \sum_{k=1}^{i-1}\abs{z^{\set{k,i}}-z^{\set{k,j}}}^p + \abs{\gamma+z^{\set{i,j}}}^p \notag\\
& \qquad\qquad+ \sum_{k=i+1}^{j-1}\abs{z^{\set{k,j}}}^p + \gamma^p \notag \\
&\leq (i-1)\beta^p + (\gamma+z^{\set{i,j}})^p + (j-1-i)\beta^p+\gamma^p \notag \\
&= (j-2)\beta^p + \gamma^p + (\gamma + z^{\set{i,j}})^p \label{1}\\
\intertext{and}
\norm{\vp_i(\vz)-\vp_j(\vz)}_p^p &\geq \gamma^p + (\gamma + z^{\set{i,j}})^p\text{.} \label{2}
\end{align}

Thus
\[\fhi^{\set{i,j}}(\vz)\geq 1+z^{\set{i,j}} - \bigl((j-2)\beta^p + \gamma^p + (\gamma + z^{\set{i,j}})^p\bigr)^{1/p}\text{.}\]
Let $f(x)=1+x - \left((j-2)\beta^p + \gamma^p + (\gamma + x)^p\right)^{1/p}$ for $0\leq x\leq \beta$.
Then
\[%\begin{align*}
f'(x) %&= 1-\frac{1}{p}\left((j-2)\beta^p+\gamma^p+(\gamma+x)^p\right)^{\frac{1}{p}-1}p(\gamma+x)^{p-1}\\
= 1-\left(\frac{(\gamma+x)^p}{(j-2)\beta^p+\gamma^p+(\gamma+x)^p}\right)^{1-\frac{1}{p}} > 0\text{.}
%\\
%&> 1 - 1 = 0 \quad\text{since $\frac{1}{p}-1<0$.}
\]%\end{align*}
It follows that $f$ is strictly increasing, which gives that
\begin{align*}
\fhi^{\set{i,j}}(\vz) &\geq f\bigl(z^{\set{i,j}}\bigr)\geq f(0) = 1-\bigl((j-2)\beta^p+2\gamma^p\bigr)^{1/p}\\
&\geq 1- \bigl((d-2)\beta^p+2\gamma^p\bigr)^{1/p}\text{.}
\end{align*}
If we require that
\begin{equation}\label{3}
(d-2)\beta^p + 2\gamma^p = 1
\end{equation}
then $\fhi^{\set{i,j}}(\vz)\geq 0$ for all $\vz\in I$.

Also,
\begin{align*}
\fhi^{\set{i,j}}(\vz) &\leq 1+z^{\set{i,j}} - \frac{1}{R}\norm{\vp_i(\vz)-\vp_j(\vz)}_p\\
&\leq 1+z^{\set{i,j}} -\frac{1}{R}\left(\gamma^p+(\gamma+z^{\set{i,j}})^p\right)^{1/p}\text{.}
\end{align*}
Let $g(x)= 1+x -\frac{1}{R}\left(\gamma^p+(\gamma+x)^p\right)^{1/p}$ for $0\leq x \leq \beta$.
Then
\[%\begin{align*}
g'(x) %&= 1 - \frac{1}{R}\frac{1}{p}\bigl(\gamma^p+(\gamma+x)^p\bigr)^{\frac{1}{p}-1}p(\gamma+x)^{p-1}\\
= 1 - \frac{1}{R}\left(\frac{(\gamma+x)^p}{\gamma^p+(\gamma+x)^p}\right)^{1-\frac{1}{p}} > 1 - \frac{1}{R} > 0\text{.}
\]%\end{align*}
Therefore, $g$ is strictly increasing, which gives that
\[ \fhi^{\set{i,j}}(\vz) \leq g\bigl(z^{\set{i,j}}\bigr)\leq g(\beta) = 1 + \beta -\frac{1}{R}\bigl(\gamma^p+(\gamma+\beta)^p\bigr)^{1/p}\text{.} \]
In order to conclude that $\fhi^{\set{i,j}}(\vz)\leq\beta$, it is sufficient to require that
\begin{equation}\label{4}
\gamma^p+(\gamma+\beta)^p \geq R^p.
\end{equation}

Suppose for the moment that we can find $\beta,\gamma>0$ such that \eqref{3} and \eqref{4} are satisfied.
Then $\fhi\colon I\to I$ is well defined, and by Brouwer's fixed point theorem $\fhi$ has a fixed point, that is, for some $\vz_0\in I$ we have $\fhi(\vz_0)=\vz_0$, which implies that $\setbuilder{\vp_i(\vz_0)}{i\in[d]}$ is $1$-equilateral.
%Write $\vp_i:=\vp_i(\vz_0)$.
Since $\vp_i^{(i)}(\vz_0)=-\gamma$, we have to scale each $\vp_i(\vz_0)$ by $-\gamma$.
Set $\gamma=1/\lambda$ and~$\beta=\gamma\epsi$.
Then $\setbuilder{-(1/\gamma)\vp_i(\vz_0)}{i\in[d]}$ is $\lambda$-equilateral,
the requirement \eqref{3} becomes the definition of $\lambda$, namely
\[ (d-2)\epsi^p + 2 = \lambda^p\]
and the requirement \eqref{4} becomes
\begin{equation}\label{ten}
\frac{1+(1+\epsi)^p}{2+(d-2)\epsi^p} \geq R^p\text{.}
\end{equation}
It remains to verify \eqref{ten} given that $0<\epsi\leq (2(d-2))^{-1/(p-1)}$ and $R^p = 1+\frac{p-1}{2}\epsi$.
Since \[(d-2)\epsi^p=(d-2)\epsi^{p-1}\epsi\leq\epsi/2\] and $\epsi\leq 2^{-1/(p-1)}<1$, it is sufficient to show that
\[ f(\epsi) := 1+(1+\epsi)^p-\left(1+\frac{p-1}{2}\epsi\right)\left(2+\frac{\epsi}{2}\right) \geq 0\]
for all $\epsi\in[0,1]$.
A calculation gives that $f(0)=0$, $f'(0)=1/2$, $f''(\epsi)=p(p-1)(1+\epsi)^{p-2}-(p-1)/2$, $f''(0)=(p-\frac{1}{2})(p-1)>0$ and $f''(1)=\frac{p-1}{2}(p2^{p-1}-1)>0$.
Thus $f''$ is monotone and positive at the endpoints of $[0,1]$, hence positive on the whole $[0,1]$, and it follows that $f'$ is positive on $[0,1]$ and $f(\epsi)\geq 0$ for all $\epsi\in[0,1]$.
%First consider the case $p\geq 2$.
%Then $(1+\epsi)^p > 1+p\epsi + \frac{p(p-1)}{2}\epsi^2$, and it is thus sufficient to show that
%\begin{equation}\label{ineqp2}
%\frac{2+p\epsi + \frac{p(p-1)}{2}\epsi^2}{2+(d-2)\epsi^p} \geq 1+\frac{p-1}{2}\epsi\text{.}
%\end{equation}
%However,
%\begin{align*}
%&\quad 2+p\epsi + \frac{p(p-1)}{2}\epsi^2 - \left(2+(d-2)\epsi^p\right)\left(1+\frac{p-1}{2}\epsi\right)\\
%&= \left(1-(d-2)\epsi^{p-1}\right)\left(\epsi + \frac{1}{2}(p-1)\epsi^2\right)+\frac{1}{2}(p-1)^2\epsi^2\\
%&> 0\quad\text{since $1-(d-2)\epsi^{p-1}>0$,}
%\end{align*}
%and \eqref{ineqp2} follows.
%
%Next consider the remaining case $1< p < 2$.
%Then we only have the bound $(1+\epsi)^p\geq 1+p\epsi$.
%On the other hand, since $d\geq 3$ and $p<2$, it follows that $\epsi<1/2$.
%It is now sufficient to show that
%\begin{equation}\label{ineqp1}
%\frac{2+p\epsi}{2+(d-2)\epsi^p} \geq 1+\frac{p-1}{2}\epsi\text{.}
%\end{equation}
%A calculation similar to the previous one gives
%\begin{align*}
%&\quad 2+p\epsi - \left(2+(d-2)\epsi^p\right)\left(1+\frac{p-1}{2}\epsi\right)\\
%&= \left(1-(d-2)\epsi^{p-1}\right)\left(\epsi + \frac{1}{2}(p-1)\epsi^2\right)-\frac{1}{2}(p-1)\epsi^2\\
%&\geq \frac{1}{2}\left(\epsi+\frac{1}{2}(p-1)\epsi^2\right)-\frac{1}{2}(p-1)\epsi^2\\
%&= \epsi\left(\frac{1}{2}-\frac{1}{4}(p-1)\epsi\right)
%> \epsi\left(\frac{1}{2}-\frac{1}{4}\epsi\right) > 0\quad\text{since $\epsi<2$,}
%\end{align*}
%and \eqref{ineqp1} follows.
\end{proof}
\begin{proof}[Proof of Theorem~\ref{thm:approxlp}]
Suppose that the theorem is false.
Then for some fixed $p\in(1,\infty)$ and $d\geq 3$ and for all $c>1$, there exists a $d$\dimensional~$X$ such that $\BMdist{X}{\ell_p^d}< c$ and $m(X)\geq d+2$.
Choose a sequence $X_n=(\R^d,\norm{\cdot}_{(n)})$ such that $m(X_n)\geq d+2$ and
\[ \norm{\vx}_{(n)}\leq\norm{x}_p\leq \left(1+\frac{1}{n}\right)^{1/p}\norm{\vx}_{(n)}\quad\text{for all $\vx\in\R^d$.}\]
If $n$ is sufficiently large, in particular if 
\[n>\frac{2(2d-4)^{1/(p-1)}}{p-1}\text{,}\]
and if we choose $\epsi=2/(n(p-1))$, then $1/n=(p-1)\epsi/2$ and $\epsi<(d-2)^{-1/(p-1)}$, and we may apply Proposition~\ref{prop:lpapprox} to obtain an equilateral set $\setbuilder{\vp_i{(n)}}{i\in[d]}$ in $X_n$ such that $\vp_i^{(i)}(n)=1$ for all $i\in[d]$ and $-\epsi<\vp_i^{(j)}(n)\leq 0$ for all $i,j\in[d]$, $i\neq j$.
Since $m(X_n)\geq d+2$, there exist points $\vp_{d+1}(n), \vp_{d+2}(n)\in X_n$ such that $\setbuilder{\vp_i{(n)}}{i\in[d+2]}$ is equilateral.
By passing to a subsequence we may assume without loss of generality that $\vp_{d+1}(n)\to\vp$ and $\vp_{d+2}(n)\to\vq$ as~$n\to\infty$.
Since $\vp_i(n)\to\ve_i$ and $\BMdist{X_n}{\ell_p^d}\to 1$ as $n\to\infty$, it follows that $\set{\ve_{1},\dots,\ve_{d},\vp,\vq}$ is equilateral in $\ell_p^d$.
This contradicts Proposition~\ref{prop:lpcalc}.
\end{proof}

\section{A construction with Hadamard matrices}\label{section:hadamard}
In \cite{Swanepoel-AdM} Hadamard matrices were used to construct equilateral sets in $\ell_p^d$ of cardinality greater than $d+1$, for all $p\in(1,2)$ and sufficiently large $d$ depending on $p$.
The construction used here is a more involved version of this idea.
Before introducing the properties of Hadamard matrices that will be needed, we first consider a special case to illustrate the construction.
\begin{lemma}\label{lemma:lpplane}
Let $1\leq p\leq 2$.
For each $\lambda\in [2^{1-1/p}, 2^{1/p}]$ there exist linearly independent unit vectors $\vu, \vv\in\ell_p^2$ such that $\norm{\vu+\vv}_p=\norm{\vu-\vv}_p=\lambda$.
\end{lemma}
\begin{proof}
Let $\vu=(\alpha,\beta)$ and $\vv=(-\beta,\alpha)$ where $\alpha,\beta\geq 0$ and $\alpha^p+\beta^p=1$.
Then $\norm{\vu\pm\vv}_p^p=\abs{\alpha+\beta}^p+\abs{\alpha-\beta}^p$, which ranges from $2$ when $\alpha=0$ and $\beta=1$, to $2^{p-1}$ when $\alpha=\beta=2^{-1/p}$.
\end{proof}

\begin{lemma}[Monotonicity lemma]\label{lemma:monotonicity}
Let $\vu$ and $\vv$ be linearly independent unit vectors in a strictly convex $2$\dimensional normed space.
Let $\vp\neq\vo$ be any point such that $\vu$ is between~$\frac{1}{\norm{\vp}}\vp$ and $\vv$ on the boundary of the unit ball.
Then $\norm{\vp-\vu}<\norm{\vp-\vv}$.
\end{lemma}
For a proof, see \cite[Proposition~31]{Martini2001}.
For non-strictly convex norms the above lemma still holds with a non-strict inequality.
On the other hand, the following corollary of the monotonicity lemma is false when the norm is not strictly convex, as easy examples show.
\begin{lemma}\label{lemma:scplane}
Let $\vu$ and $\vv$ be linearly independent unit vectors in a strictly convex $2$\dimensional normed space.
Suppose that $\vx$ is such that $\norm{\vx-\vu}=\norm{\vx+\vu}$ and $\norm{\vx-\vv}=\norm{\vx+\vv}$.
Then~$\vx=\vo$.
\end{lemma}
\begin{proof}
Without loss of generality, $\vx=\alpha\vu+\beta\vv$ with $\alpha,\beta\geq 0$.
If $\vx\neq\vo$, then by Lemma~\ref{lemma:monotonicity},
\[\norm{\vx-\vv}<\norm{\vx+\vu}=\norm{\vx-\vu}<\norm{\vx+\vv},\]
a contradiction.
\end{proof}

\begin{proposition}\label{prop17}
Let $X$ be any normed space, $d\geq 4$, $q\in[1,\infty)$, and $1\leq p < \frac{\log (5/2)}{\log 2}$.
Then $m(\ell_p^d\oplus_q X)\leq 5$.
If $p =\frac{\log (5/2)}{\log 2}$, then $m(\ell_p^d\oplus_q X)\leq 6$.
\end{proposition}
\begin{proof}
Consider the following subset of $\ell_p^d\oplus_q X = \left(\ell_p^4\oplus_p\ell_p^{d-4}\right)\oplus_q X$:
\begin{equation*}
    \begin{array}{r@{}r@{\,}r@{\,}r@{\;\;}r@{\;\;}l}
S= \{\; (& 1, & 1, & 1, &0, &\vo,\; \vo\;),\\
        (& 1, &-1, &-1, &0, &\vo,\; \vo\;),\\
        (&-1, & 1, &-1, &0, &\vo,\; \vo\;),\\
        (&-1, &-1, & 1, &0, &\vo,\; \vo\;),\\
        (& 0, & 0, & 0, &\lambda, &\vo,\; \vo\;)\; \}.
     \end{array}
\end{equation*}
By setting $\lambda=(2^{p+1}-3)^{1/p}$, the set $S$ becomes $2^{1+1/p}$-equilateral.
We show that $S$ is maximal equilateral if $p<\log(5/2)/\log 2$ and can be uniquely extended if $p=\log(5/2)/\log 2$.

Suppose that $(\alpha_1,\alpha_2,\alpha_3,\alpha_4,\vx,\vy)$ has distance $2^{1+1/p}$ to each point in $S$.
Then $(\alpha_1,\alpha_2,\alpha_3)$ has the same distance in $\ell_p^3$ to the points
\[(1,1,1), (1,-1,-1), (-1,1,-1), (-1,-1,1)\text{,} \]
which gives \[\norm{(\alpha_1,\alpha_2)-(1,1)}_p=\norm{(\alpha_1,\alpha_2)-(-1,-1)}_p\] and \[\norm{(\alpha_1,\alpha_2)-(1,-1)}_p=\norm{(\alpha_1,\alpha_2)-(-1,1)}_p.\]
It follows (from Lemma~\ref{lemma:scplane} if $p>1$) that $(\alpha_1,\alpha_2)=(0,0)$.
Thus $\abs{\alpha_3-1}=\abs{\alpha_3+1}$, which gives $\alpha_3=0$ and $3+\abs{\alpha_4}^p=\abs{\alpha_4-\lambda}^p$.
By Lemma~\ref{lemma:calc}, the function $f(x)=3+\abs{x}^p-\abs{x-\lambda}^p$ is increasing (strictly increasing if $p>1$).
Since $f(\alpha_4)=0$ and $f(-\lambda)=2^{p+1}(\frac{5}{2}-2^p)\geq0$ ($>0$ if $p=1$), it follows that $\alpha_4\leq-\lambda$.
Then by assumption,
\begin{align*}
2^{1+1/p} &=\norm{(0,0,0,\alpha_4,\vx,\vy)-(1,1,1,0,\vo,\vo)}\\
&=\left((3+\abs{\alpha_4}^p+\norm{\vx}_p^p)^{q/p}+\norm{\vy}^q\right)^{1/q} \\
&\geq (3+\lambda^p)^{1/p} = 2^{1+1/p},
\end{align*}
and equality holds throughout, which implies that $\alpha_4=-\lambda$, $\vx=\vo$ and $\vy=\vo$.
Also, by assumption, $2^{1+1/p}=\norm{(0,0,0,-\lambda,\vo,\vo)-(0,0,0,\lambda,\vo,\vo)}=2\lambda$, which implies $p=\frac{\log (5/2)}{\log 2}$. 
Therefore, $S$ is a maximal equilateral set unless $p=\frac{\log (5/2)}{\log 2}$, and then $S\cup\set{(0,0,0,-\lambda,\vo,\vo)}$ is maximal.
\end{proof}

An $n\times n$ matrix $H$ is called a \emph{Hadamard matrix} of order $n$ if each entry equals $\pm 1$ and $HH^{\Tr}=nI$.
It is well known \cite[Chapter~18]{vLW} that if an Hadamard matrix of order $n$ exists, then $n=1$, $n=2$ or $n$ is divisible by~$4$.
It is conjectured that there exist Hadamard matrices of all orders divisible by $4$.
This is known for all multiples of $4$ up to $664$ \cite{KT}.
The next lemma summarises the only results (all well known) on the existence of Hadamard matrices that we will use.
\begin{lemma}\label{hadamard}
\mbox{}
\begin{enumerate}
\item\label{20.1} There exist Hadamard matrices of orders $1$, $2$, $4$, $8$, $12$.

\item\label{20.2} Let $x\geq 1$.
The interval $(x,2x)$ contains the order of some Hadamard matrix iff $x\notin\set{1,2,4}$.

\item\label{20.3} Let $H(x)$ be the largest order $n$ of an Hadamard matrix with $n<x$.
Then \[\lim_{x\to\infty} H(x) / x = 1\text{.}\]
\end{enumerate}
\end{lemma}
\begin{proof}
\eqref{20.1} is well-known.

\eqref{20.2}
Given Hadamard matrices $H_1$ of order $n_1$ and $H_2$ of order $n_2$, the Kronecker product $H_1\otimes H_2$ is an Hadamard matrix of order $n_1n_2$ \cite[Chapter~18]{vLW}.
Starting with the (essentially unique) Hadamard matrices of orders $2$ and $12$, we obtain Hadamard matrices of orders $2^k$ and $6\cdot 2^k$ for all $k\geq 1$.
This is sufficient to cover every interval $(x,2x)$ except for $(1,2)$, $(2, 4)$ and $(4,8)$.

\eqref{20.3}
The Paley construction \cite[Chapter~18]{vLW} gives an Hadamard matrix of order $q+1$ for any prime power $q\equiv 3\pmod{4}$.
By the prime number theorem for arithmetic progressions \cite{Jameson} the number of primes less than $x$ that are congruent to $3$ modulo $4$ is $(1+o(1))x/(2\ln x)$.
This implies that the largest such prime less than $x$ is $\geq (1+o(1))x$, which gives $H(x)/x\to 1$ as~$x\to\infty$.
\end{proof}

An Hadamard matrix is \emph{normalised} if its first column are all $+1$s.
Note that any Hadamard matrix can be normalised by multiplying each row by its first entry.
If
\[ H = \begin{bmatrix} 1 & \vh_1 \\ 1 & \vh_2\\ \vdots & \vdots\\ 1 & \vh_n\end{bmatrix} \]
is a normalised Hadamard matrix we say that $\set{\vh_1,\dots,\vh_n}\subset\R^{n-1}$ is a \emph{Hadamard simplex}.
In the sequel we repeatedly use the fact that each $\vh_i$ has $n-1$ coordinates, each $\pm 1$, and that any two $\vh_i$ differ in exactly $n/2$ coordinates.
In particular, the distance between any two vertices of an Hadamard simplex in $\ell_p^{n-1}$ equals $n^{1/p}2^{1-1/p}$ and all vertices lie on a sphere with centre $\vo$ and radius $(n-1)^{1/p}$.
The next lemma shows in particular (by taking $X=\R$) that an Hadamard simplex cannot lie on any other sphere of $\ell_p^{n-1}$ if $p\in[1,\infty)$.
\begin{lemma}\label{lemma:hadamard}
Let $\set{\vh_1,\dots,\vh_n}\subset\R^{n-1}$ be an Hadamard simplex, $p\in[1,\infty)$, $X$ a normed space and $\vu\in X$.
Suppose that
\[(\vx_1,\dots,\vx_{n-1})\in \overbrace{X\oplus_p\dots\oplus_p X}^{\text{$n-1$ summands}}\]
has the same distance to each point of $\setbuilder{\vh_i\otimes \vu}{i\in[n]}\subset X\oplus_p\dots\oplus_p X$.
Then $\norm{\vx_i-\vu}=\norm{\vx_i+\vu}$ for all $i\in[n]$.
\end{lemma}
\begin{proof}
Let $\vh_i=[h_{i,1},h_{i,2},\dots,h_{i,n-1}]$ for $i\in[n]$.
We may assume without loss of generality that $\vh_1=[-1,-1,\dots,-1]$.
Since $\vx=(\vx_1,\vx_2,\dots,\vx_{n-1})$ is equidistant to all $\vh_i\otimes\vu$, there exists $D\geq 0$ such that $\sum_{j=1}^{n-1}\norm{\vx_j-h_{i,j}\vu}^p=D^p$ for each $i\in[n]$.
Subtract the first of these equations from the others to obtain the system
\begin{equation}\label{system}
\begin{bmatrix} 
    \vh_2-\vh_1 \\ \vh_3-\vh_1 \\ \vdots \\ \vh_{n}-\vh_1
\end{bmatrix} 
\begin{bmatrix} 
    \norm{\vx_1-\vu}^p-\norm{\vx_1+\vu}^p \\
    \norm{\vx_2-\vu}^p-\norm{\vx_2+\vu}^p \\
    \vdots \\
    \norm{\vx_{n-1}-\vu}^p-\norm{\vx_{n-1}+\vu}^p \\
\end{bmatrix} = 
\begin{bmatrix} 
    0 \\ 0 \\ \vdots \\ 0 
\end{bmatrix}
\end{equation}
The Hadamard matrix $H$ is invertible.
If we subtract the first row from all the other rows, the resulting matrix
\[ 
\begin{bmatrix} 
    1 & \vo \\ 0 & \vh_2-\vh_1 \\ \vdots & \vdots \\ 0 & \vh_{n}-\vh_1 
\end{bmatrix} 
\]
is still invertible.
It follows that \eqref{system} has the unique solution 
\[\norm{\vx_j-\vu}^p-\norm{\vx_j+\vu}^p=0\quad\text{for all $j\in[n-1]$.}\qedhere\]
\end{proof}

\begin{lemma}\label{lemma:equid}
Let $\set{\vh_1,\dots,\vh_n}\subset\R^{n-1}$ be an Hadamard simplex and $p\in[1,\infty)$.
Let $\vu$ and $\vv$ be linearly independent unit vectors in a strictly convex $2$\dimensional normed space $X$.
Suppose that 
\[\vx=(\vx_1,\dots,\vx_{n-1})\in \overbrace{X\oplus_p\dots\oplus_p X}^{\text{$n-1$ summands}}\]
has the same distance in the $p$-norm to each point of $\setbuilder{\vh_i\otimes \vu}{i\in[n]}$, and the same distance to each point of $\setbuilder{\vh_i\otimes\vv}{i\in[n]}$.
Then $\vx=\vo$.
\end{lemma}
\begin{proof}
Combine Lemmas~\ref{lemma:scplane} and \ref{lemma:hadamard}.
\end{proof}
\begin{proposition}\label{prop20}
Let $p\in[1,2)$, $q\in[1,\infty)$, and $X$ any normed space.
Let $k_1,k_2\in\N$ be such that there exist Hadamard matrices of orders $k_1$ and $k_2$ and such that
\begin{gather}
2-2^{p-1} < \frac{1}{k_1}+\frac{1}{k_2} < 4-2^p\text{,} \label{one}\\
(1-2^{-p})(2-2^{p-1}) < (1-2^{1-p})\frac{1}{k_1}+\frac{1}{k_2}\text{,}\label{two}\\
\text{and }(1-2^{-p})(2-2^{p-1}) < \frac{1}{k_1}+(1-2^{1-p})\frac{1}{k_2}\text{.}\label{three}
%\text{and if $k_1= k_2$, then $2-2^{p-1 }< \frac{1}{k_1}+\frac{1}{k_2}< 4-2^p$.}\label{four}
\end{gather}
Then $m(\ell_p^{d}\oplus_q X)\leq 2(k_1+k_2)$ for all $d\geq 2(k_1+k_2-1)$.
\end{proposition}
\begin{proof}
It is sufficient to construct an equilateral set $S$ of cardinality $2(k_1+k_2)$ in $\ell_p^{2(k_1+k_2-1)}$ that does not lie on any sphere.
Then $(S\oplus\set{\vo})\oplus\{\vo\}$ will be maximal equilateral in $\ell_p^d\oplus_q X = \left(\ell_p^{2(k_1+k_2-1)}\oplus_p\ell_p^{d-2(k_1+k_2-1)}\right)\oplus_q X$ for any $q\in[1,\infty)$.

Let $\alpha_1,\alpha_2,\lambda_1,\lambda_2\in\R$ (to be fixed later) such that
\begin{equation}\label{eq:constraints}
\alpha_1,\alpha_2 > 0\quad\text{and}\quad 2^{1-1/p}\leq\lambda_1,\lambda_2\leq2^{1/p}\text{.}
\end{equation}
By Lemma~\ref{lemma:lpplane} there exist unit vectors $\vu_1,\vu_2,\vv_1,\vv_2\in\ell_p^2$ such that $\set{\vu_i,\vv_i}$ is linearly independent and $\norm{\vu_i\pm\vv_i}_p=\lambda_i$ for $i=1,2$.
Let $\setbuilder{\vg_i}{i\in[k_1]}\subset\ell_p^{k_1-1}$ and $\setbuilder{\vh_i}{i\in[k_2]}\subset\ell_p^{k_2-1}$ be Hadamard simplices.
Consider the following subsets of $\ell_p^{2(k_1+k_2-1)}=\R\oplus_p\ell_p^{2(k_1-1)}\oplus_p\R\oplus_p\ell_p^{2(k_2-1)}$:
\begin{equation*}
    \begin{array}{l@{\;=\bigl\{\bigl(}r@{\, ,\;\,}r@{\, ,\;\,}r@{\, ,\;\,}r@{\bigr)\;\colon\,}r@{\bigr\}}l@{}l}
S_1^- & -\alpha_1 & k_1^{-1/p}\vg_i\otimes\vu_1 &0 &\vo & i\in[k_1] & \text{,} \\[1mm]
S_1^+ & \alpha_1  & k_1^{-1/p}\vg_i\otimes\vv_1 &0 &\vo & i\in[k_1] & \text{,} \\[1mm]
S_2^- & 0         &\vo                          &-\alpha_2 & k_2^{-1/p}\vh_i\otimes\vu_2 &i\in[k_2] & \text{,} \\[1mm]
S_2^+ & 0         &\vo &\alpha_2                & k_2^{-1/p}\vh_i\otimes\vv_2 & i\in[k_2] & \text{.}
     \end{array}
\end{equation*}
We would like to choose $\alpha_1,\alpha_2,\lambda_1,\lambda_2$ so as to make $S=S_1^-\cup S_1^+ \cup S_2^-\cup S_2^+$ equilateral and non-spherical.
Note that then $\card{S}=2(k_1+k_2)$, since all the points will be distinct.

The $p$th power of the distance between points
\begin{itemize}
\item in the same set $S_1^{\pm}$ equals \[\frac{k_1}{2}\frac{1}{k_1}\norm{2\vu_1}_p^p=\frac{k_1}{2}\frac{1}{k_1}2^p=2^{p-1}\text{,}\]
\item in the same set $S_2^{\pm}$ similarly equals $\frac{k_2}{2}\frac{1}{k_2}2^p=2^{p-1}$,
\item in $S_1^-$ and $S_1^+$ equals \[(2\alpha_1)^p+(k_1-1)\frac{1}{k_1}\norm{\vu_1\pm\vv_1}_p^p=(2\alpha_1)^p+\Bigl(1-\frac{1}{k_1}\Bigr)\lambda_1^p\text{,}\]
\item in $S_2^-$ and $S_2^+$ similarly equals $(2\alpha_2)^p+(1-\frac{1}{k_2})\lambda_2^p$,
\item in $S_1^-\cup S_1^+$ and $S_2^-\cup S_2^+$ equals
\[\alpha_1^p+\frac{k_1-1}{k_1}+\alpha_2^p+\frac{k_2-1}{k_2}=\alpha_1^p+\alpha_2^p+2-\Bigl(\frac{1}{k_1}+\frac{1}{k_2}\Bigr)\text{.}\]
\end{itemize}
For $S$ to be equilateral, we therefore need
\begin{gather}
(2\alpha_1)^p+\Bigl(1-\frac{1}{k_1}\Bigr)\lambda_1^p=2^{p-1},\quad 
(2\alpha_2)^p+\Bigl(1-\frac{1}{k_2}\Bigr)\lambda_2^p=2^{p-1}\label{eq:equil1}\\
\intertext{and}
\alpha_1^p+\alpha_2^p+2-\Bigl(\frac{1}{k_1}+\frac{1}{k_2}\Bigr)=2^{p-1}.\label{eq:equil2}
\end{gather}
The set $S$ will lie on some sphere iff some $(\beta,\vx,\gamma,\vy)$ is equidistant to $S$.
This implies that $\vx$ is equidistant to all $k_1^{-1/p}\vg_i\otimes\vu_1$ and also equidistant to all $k_1^{-1/p}\vg_i\otimes\vv_1$.
By Lemma~\ref{lemma:equid}, $\vx=\vo$.
Similarly, $\vy=\vo$.
Then $\abs{-\alpha_1-\beta}=\abs{\alpha_1-\beta}$, which gives $\beta=0$.
Similarly, $\gamma=0$.
Thus $S$ can only lie on a sphere with centre $\vo$ and radius $(\alpha_1^p+(k_1-1)/k_1)^{1/p}=(\alpha_2^p+(k_2-1)/k_2)^{1/p}$.
%It follows that $S$ lies on a sphere iff $\alpha_1=\alpha_2$.
Therefore, for $S$ not to lie on a sphere, we need
\begin{equation}\label{eq:inequation}
\alpha_1^p-\frac{1}{k_1}\neq\alpha_2^p-\frac{1}{k_2}.
\end{equation}
It turns out that the hypotheses \eqref{one}, \eqref{two}, \eqref{three} are sufficient
for the three simultaneous equations \eqref{eq:equil1} and \eqref{eq:equil2} to have a solution in $\alpha_1,\alpha_2,\lambda_1,\lambda_2$ given the constraints \eqref{eq:constraints} and \eqref{eq:inequation}.
This can be seen as follows.
First use \eqref{eq:equil1} to eliminate $\alpha_1$ and $\alpha_2$ from \eqref{eq:constraints}, \eqref{eq:equil2} and \eqref{eq:inequation}, and set $x_1=(1-\frac{1}{k_1})\lambda_1^p$ and $x_2=(1-\frac{1}{k_2})\lambda_2^p$ to obtain that the simultaneous solvability of \eqref{eq:constraints}, \eqref{eq:equil1}, \eqref{eq:equil2} and \eqref{eq:inequation} is equivalent to the existence of $x_1,x_2\in\R$ such that
\begin{align}
x_1+x_2&=2^p\Bigl(3 - 2^{p-1}-\frac{1}{k_1}-\frac{1}{k_2}\Bigr)\text{,}\label{19}\\
2^{p-1}\bigl(1-\frac{1}{k_1}\bigr) &\leq x_1 \leq 2\Bigl(1-\frac{1}{k_1}\Bigr)\text{,}\label{20}\\
2^{p-1}\bigl(1-\frac{1}{k_2}\bigr) &\leq x_2\leq 2\Bigl(1-\frac{1}{k_2}\Bigr)\text{,}\label{21}\\
x_1,x_2 &< 2^{p-1}\text{,}\label{21a}\\
\qquad\text{and}\quad x_1-x_2&\neq 2^p\left(\frac{1}{k_2}-\frac{1}{k_1}\right)\text{.}\label{22}
\end{align}
Geometrically, it is sufficient to show that the line~$\ell$ in the $x_1x_2$-plane with equation \eqref{19} intersects the interior of the rectangle $R$ with bottom-left corner $\va=2^{p-1}(1-\frac{1}{k_1},1-\frac{1}{k_2})$ and top-right corner 
\[\vb=\left(\min\Bigset{2^{p-1},2\Bigl(1-\frac{1}{k_1}\Bigr)}, \min\Bigset{2^{p-1},2\Bigl(1-\frac{1}{k_2}\Bigr)}\right)\text{.}\]
Indeed, if $\ell$ intersects the interior of $R$ then it intersects $R$ in infinitely many points, among which one can be chosen that satisfies \eqref{22}.
That $\ell$ intersects the interior of $R$ is equivalent to $\va$ lying below $\ell$:
\begin{equation}\label{23}
2^{p-1}\Bigl(1-\frac{1}{k_1}\Bigr)+2^{p-1}\Bigl(1-\frac{1}{k_2}\Bigr) < 2^p\left(3 - 2^{p-1}-\frac{1}{k_1}-\frac{1}{k_2}\right)
\end{equation}
and $\vb$ lying above $\ell$:
\begin{equation}\label{24}
2^p\Bigl(3 - 2^{p-1}-\frac{1}{k_1}-\frac{1}{k_2}\Bigr) < \min\Bigset{2^{p-1},2\Bigl(1-\frac{1}{k_1}\Bigr)} + \min\Bigset{2^{p-1},2\Bigl(1-\frac{1}{k_2}\Bigr)}\text{.}
\end{equation}
Inequality~\eqref{23} is equivalent to the second inequality in \eqref{one}.

Inequality \eqref{24} can be split into four simultaneous inequalities:
\begin{align}
\label{a} 2^p\Bigl(3 - 2^{p-1}-\frac{1}{k_1}-\frac{1}{k_2}\Bigr) & < 2\Bigl(1-\frac{1}{k_1}\Bigr)+2\Bigl(1-\frac{1}{k_2}\Bigr)\text{,} \tag{\theequation a}\\
\label{b} 2^p\Bigl(3 - 2^{p-1}-\frac{1}{k_1}-\frac{1}{k_2}\Bigr) & < 2^{p-1}+2^{p-1}\text{,} \tag{\theequation b}\\
\label{c} 2^p\Bigl(3 - 2^{p-1}-\frac{1}{k_1}-\frac{1}{k_2}\Bigr) & < 2\Bigl(1-\frac{1}{k_1}\Bigr)+2^{p-1}\text{,} \tag{\theequation c}\\
\label{d} 2^p\Bigl(3 - 2^{p-1}-\frac{1}{k_1}-\frac{1}{k_2}\Bigr) & < 2^{p-1}+2\Bigl(1-\frac{1}{k_2}\Bigr)\text{.} \tag{\theequation d}
\end{align}
Each of \eqref{a} and \eqref{b} is equivalent to the first inequality in \eqref{one}.
Inequality~\eqref{c} is equivalent to \eqref{two} and \eqref{d} is equivalent to \eqref{three}.

It follows that $\alpha_1,\alpha_2,\lambda_1,\lambda_2$ can be chosen so that $S$ is equilateral and non-spherical.
\end{proof}

\begin{proof}[Proof of Theorem~\ref{thm:lp}]
By Lemma~\ref{hadamard} the interval $\bigl(2/(4-2^p),4/(4-2^p)\bigr)$ contains the order of some Hadamard matrix except if $2/(4-2^p)\in\set{1,2,4}$.
Therefore, unless \[p\in\set{1,\frac{\log3}{\log2}, \frac{\log(7/2)}{\log 2}}\text{,}\]
if we let $k_1=k_2=k$ be the order of such an Hadamard matrix, conditions \eqref{one}, \eqref{two} and \eqref{three} are satisfied, and Proposition~\ref{prop20} gives a maximal equilateral set of size $C(p)=2(k_1+k_2)=4k$ in all dimensions at least $d_0(p)=2(k_1+k_2-1)=4k-2$.

Asymptotically when $p\to 2$, we may choose $k$ to be the largest order of an Hadamard matrix such that $k<4/(4-2^p)$.
Then by Lemma~\ref{hadamard}, $4k < 16/(4-2^p)\sim4/((2-p)\ln 2)$, which gives the required asymptotic upper bounds for $C(p)$ and $d_0(p)$.

The exceptional case $p=1$ has already been dealt with in Proposition~\ref{prop17}.
For $p=\log3/\log 2$ we may use $(k_1,k_2)=(2,4)$ in Proposition~\ref{prop20} and for $p=\log(7/2)/\log 2$ we may use $(k_1,k_2)=(4,8)$.
%That the conditions of Proposition~\ref{prop20} are satisfied in these special cases is easily checked.
This covers the existence of $C(p)$ and $d_0(p)$ for all $p\in[1,2)$.

The last column of Table~\ref{table} indicates how each line in that table is obtained: Proposition~\ref{prop17} covers the case $1\leq p \leq \tfrac{\log (5/2)}{\log 2}$, and in the remaining cases Proposition~\ref{prop20} is applied with Hadamard matrices of various orders $k_1$ and $k_2$.
\end{proof}

\begin{proof}[Proof of Corollary~\ref{cor:lp}]
%We take $X$ to be the zero space in Theorem~\ref{thm:lp}.
Let $p\in[1,2)$ and let $\Gamma$ be any set.
If $\Gamma$ is infinite, then by Theorem~\ref{thm:lp} where $X$ is chosen to be $\ell_p(\Gamma)$ and $q=p$, $m(\ell_p^d\oplus_p\ell_p(\Gamma))\leq C(p)$ for large enough $d$.
However, since $\Gamma$ is infinite, $\ell_p^d\oplus_p\ell_p(\Gamma)$ is isometrically isomorphic to $\ell_p(\Gamma)$.

For the remaining case where $\Gamma$ is finite, it is sufficient to bound $m(\ell_p^d)$ for all $d$.
Again by Theorem~\ref{thm:lp} where $X$ is chosen to be the zero space, $m(\ell_p^d)\leq C(p)$ for all $d\geq d_0(p)$.
For $d<d_0(p)$, we apply the Danzer-Gr\"unbaum-Petty-Soltan result to obtain $m(\ell_p^d) < 2^{d_0(p)}$.

Thus for fixed $p\in[1,2)$, we have found an upper bound to $m(\ell_p(\Gamma))$ uniform over all $\Gamma$.
\end{proof}

\section*{Acknowledgements}
We thank Roman Karasev for helpful remarks.
We are very grateful to two anonymous referees whose careful proofreading caught a few egregious errors.

\end{document}